\documentclass[11pt,letterpaper]{amsart}

\usepackage{amsmath,amssymb,amsthm}
\usepackage{mathrsfs}
\usepackage{indentfirst,color}
\usepackage[colorlinks,citecolor=red,linkcolor=blue,urlcolor=cyan]{hyperref}
\usepackage{graphicx}
\usepackage{tikz,tikz-cd}
\usepackage{bm}

\theoremstyle{plain}%
\newtheorem{theorem}{Theorem}[section]%
\newtheorem{lemma}[theorem]{Lemma}%

\theoremstyle{definition}%
\newtheorem{definition}[theorem]{Definition}%
\theoremstyle{remark}%
\newtheorem{remark}[theorem]{Remark}%

\newcommand{\NN}{{\mathbb{N}}}%

\makeatletter
\newsavebox{\@brx}
\newcommand{\llangle}[1][]{\savebox{\@brx}{\(\m@th{#1\langle}\)}%
    \mathopen{\copy\@brx\kern-0.5\wd\@brx\usebox{\@brx}}}
\newcommand{\rrangle}[1][]{\savebox{\@brx}{\(\m@th{#1\rangle}\)}%
    \mathclose{\copy\@brx\kern-0.5\wd\@brx\usebox{\@brx}}}
\makeatother

\newcommand{\Id}{{\operatorname{Id}}}%
\newcommand{\Vol}{{\operatorname{Vol}}}%
\newcommand{\loc}{\textup{loc}}%
\newcommand{\nd}{\textup{nd}}%
\newcommand{\image}{\textup{i}}%
\newcommand{\sgn}{\textup{sgn}}%

\newcommand{\inner}[1]{\langle#1\rangle}%
\renewcommand{\bar}[1]{\overline{#1}}%
\newcommand{\ddbar}{\image\partial\bar{\partial}}%

\renewcommand{\leq}{\leqslant}%
\renewcommand{\geq}{\geqslant}%

\numberwithin{equation}{section}

\begin{document}

\title[Vanishing theorems for pseff line bundles]{Vanishing theorems for pseudo-effective line bundles }
\author[Xiankui Meng]{Xiankui Meng}
\address{Xiankui Meng: School of Mathematical Sciences and Key Laboratory of Mathematics and Information Networks (Ministry of Education), Beijing University of Posts and Telecommunications, Beijing 100876, China.}
\email{mengxiankui@amss.ac.cn}

\author[Chenghao Qing]{Chenghao Qing}
\address{Chenghao Qing: Yau Mathematical Sciences Center, Tsinghua University, Beijing 100084, China.}
\email{qingchenghao@amss.ac.cn}

\author[Xiangyu Zhou]{Xiangyu Zhou}
\address{Xiangyu Zhou: Institute of Mathematics, Academy of Mathematics and Systems Science, Chinese Academy of Sciences, Beijing 100190, China.}
\email{xyzhou@math.ac.cn}

\date{}

\thanks{The first author was partially supported by the National Natural Science Foundation of China (Grant No. 12271057) and by the National Key Research and Development Program of China (Grant No. 2021YFA1002600).
The second author was supported by the China Postdoctoral Science Foundation (Grant No. 2025M773087).
The third author was supported by National Key R\&D Program of China (No. 2021YFA1003100) and by the National Natural Science Foundation of China (Grant No. 12288201).}

\begin{abstract}
    In the present paper, we establish a general Kawamata-Viehweg-Koll\'ar-Nadel type vanishing theorem for higher direct images in terms of numerical dimension for closed positive currents  on compact K\"ahler manifolds, unifying a number of important vanishing theorems.
\end{abstract}



\maketitle

\section{Introduction}

Vanishing theorems for cohomology groups valued in coherent analytic sheaves play a fundamental role in the study of positivity in algebraic and analytic geometry.
There are numerous important results in this field, for example, 
the classical Kodaira's vanishing theorem, Cartan-Serre's theorem A and B, Kawamata-Viehweg's vanishing theorem, Koll\'ar's vanishing theorem, Nadel's vanishing theorem and their various generalizations, which have been playing a central role in the classification theory of higher dimensional projective algebraic varieties and the extension theory in several complex variables.

Recently, the notion of multiplier ideal sheaf associated with a plurisubharmonic function or a singular hermitian metric with semi-positive curvature current plays an important role in the study of vanishing theorems. 

Let us first recall some definitions and notions.
A singular Hermitian metric $h$ on a holomorphic line bundle $L\to X$ is simply a metric which can be expressed locally as
$e^{-\varphi_U}$ on $U$ such that $\varphi_U$ is $L^1_\loc$, where
$U\subset X$ is a local coordinate chart such that $L|_U\simeq U\times\mathbb{C}.$
It has a well-defined curvature current $\image\Theta_{L,h}:=\ddbar\varphi_U.$

If $\varphi$ is a quasi-plurisubharmonic (quasi-psh for short) function on $X$, the multiplier ideal sheaf $\mathcal{I}(\varphi)$ is the ideal subsheaf of $\mathcal{O}_X$ defined by
$$\mathcal{I}(\varphi)_x=\{f\in\mathcal{O}_{X,x}:~\exists\ U\ni x\  \text{such that}\ \int_U|f|^2e^{-\varphi}d\lambda<+\infty  \},$$
where $U$ is an open coordinate neighborhood of $x$ and $d\lambda$ is the standard Lebesgue measure in $\mathbb{C}^n.$
The upper regularized multiplier ideal sheaf $\mathcal{I}_+(\varphi)$ is defined by
$\mathcal{I}_+(\varphi)=\bigcup_{\varepsilon>0}((1+\varepsilon)\varphi)$.

It is easy to see that associated to a singular Hermitian metric $h$ on $L$ satisfying
$\image\Theta_{L,h}\geq\gamma$
for some smooth $(1,1)$-form $\gamma$ in the sense of currents,
there is a well-defined multiplier ideal sheaf $\mathcal{I}(h)$ on $X$ and $\mathcal{I}(h)$ is coherent (\cite{Nad90}).

Another basic property of the  multiplier ideal sheaf which was posed as a conjecture by Demailly (proved in \cite{GZ15-a}) is as follows.

\begin{theorem}[strong openness \cite{GZ15-a}]
The upper regularized multiplier ideal sheaf is exactly the multiplier ideal sheaf, that is, $\mathcal{I}_+(\varphi)=\mathcal{I}(\varphi)$ if $\varphi$ is a quasi-plurisubharmonic function
\end{theorem}

By using Kawamata-Viehweg vanishing theorem, one can prove the famous base point free theorem combining with the non-vanishing theorem (see \cite[\S 3.2]{KM98}).

In order to extend Kawamata-Viehweg vanishing theorem to the setting of compact K\"ahler manifolds, Demailly and Peternell in \cite[Theorem 3.3]{DP03} proved a vanishing theorem involving the upper regularized multiplier ideal sheaf $\mathcal{I}_+(\varphi)$ (therefore also true for the multiplier ideal sheaf $\mathcal{I}(\varphi)$ by the strong openness),
which can be used to derive the Kawamata-Viehweg vanishing theorem for nef line bundles with numerical dimension two.
Demailly and Peternell used their vanishing theorem to obtain a part of the abundance theorem for minimal K\"ahler threefolds.

Let's recall here a version of the Kawamata-Viehweg vanishing theorem in the context of nef line bundles and multiplier ideal sheaves.
One can see the direct proof via Nadel's vanishing theorem in \cite{DemSmall}.
\begin{theorem}[{\cite{Kaw82,Vie82,DemSmall}}]\label{K-V vanishing}
    Let $X$ be a projective algebraic manifold of dimension $n$ and let $F$ be a line bundle over $X$ such that
    some positive multiple $mF$ can be written $mF = L+D$ where $L$ is a nef line bundle and $D$ an effective divisor.
    Then
    $$H^p(X,K_X\otimes F\otimes\mathcal{I}(m^{-1}D))=0 \quad\text{for } p\geq n-\nd(L)+1,$$
    where $\nd(L)$ is the numerical dimension of $L$.
    In particular, if $F$ is a nef line bundle, then
    $$H^p(X,K_X\otimes F)=0 \quad\text{for } p\geq n-\nd(F)+1.$$
\end{theorem}

Cao developed their methods and obtained a vanishing theorem on numerical dimension valued in $\mathcal{I}_+(\varphi)$ for pseudo-effective line bundles with singular metrics in \cite[Theorem 1.4]{Cao14}.
The following Kawamata-Viehweg-Nadel type vanishing theorem on compact K\"ahler manifolds holds, by combining the results of Cao and Guan-Zhou.
\begin{theorem}[\cite{Cao14,GZ15-a}]\label{CGZ vanishing}
    Let $(L,h)$ be a pseudo-effective line bundle on a compact K\"ahler manifold $X$ of dimension $n$. Then
    $$H^p(X,K_X\otimes L\otimes \mathcal{I}(h))=0 \quad \text{for every } p\geq n-\nd(\image\Theta_{L,h})+1,$$
    where $\nd(\image\Theta_{L,h})$ is the numerical dimension of $\image\Theta_{L,h}$.
\end{theorem}
Note that if $(L,h)$ is big, then $\nd(\image\Theta_{L,h})=n$. Therefore, the above theorem for the case of big bundles reduces to Nadel vanishing theorem. One can also deduce Kawamata-Viehweg's vanishing theorem from the above vanishing theorem.\\

The motivations in this paper are two-folds, one is to consider the above types of vanishing theorems, another one is to consider the following types vanishing theorems for higher direct images. We study the relations between the two types important vanishing theorems.

The following vanishing theorem was obtained by Koll\'ar.

\begin{theorem}[{\cite[Theorem 2.1]{Kol86-a}}] \label{Kollar vanishing}
    Let $\pi:X\rightarrow Y$ be a surjective map from a smooth projective variety $X$ to a reduced projective variety $Y$.
    Let $L$ be an ample line bundle on $Y$. Then
    $$H^p(Y,R^q\pi_*K_X\otimes L)=0 \quad \text{for any } p\geq 1 \text{ and } q\geq0.$$

\end{theorem}

Koll\'ar type vanishing theorems have many important applications.
They are also closely related to the injectivity theorem and the torsion-freeness of higher direct images.
For example, we refer to \cite{FM21,FW25,Kaw85,Kol86-b,Nak87}.

Koll\'ar's vanishing theorem can be generalized to pseudo-effective line bundles via injectivity theorems with multiplier ideal sheaves as follows.

\begin{theorem}[{\cite{Fuj18,FM21}}]
    Let $\pi:X\rightarrow Y$ be a surjective holomorphic map from a compact K\"ahler manifold $X$ to a projective variety
    $Y$.
    Let $L$ be a holomorphic line bundle on $X$ equipped with a singular hermitian metric $h$.
    Let $H$ be an ample line bundle on $Y$. Assume that there exists a smooth hermitian metric $g$ on $\pi^*H$ such that
    $$\image\Theta_{\pi^*H,g}\geq 0 \quad \text{and} \quad \image\Theta_{L,h}-b\image\Theta_{\pi^*H,g}\geq 0$$
    for some $b>0$. Then
    $$H^p(Y,R^q\pi_*(K_X\otimes L\otimes\mathcal{I}(h)))=0 \quad \text{for any } p\geq 1 \text{ and } q\geq0.$$

\end{theorem}
One can see \cite{Fuj23} for the relative version.\\

The following theorem due to Ohsawa and Matsumura can be seen as a generalization of Koll\'ar's result to the complex analytic setting.
\begin{theorem}[{\cite{Mat16,Ohs84}}]
    Let $\pi:X\rightarrow Y$ be a proper surjective holomorphic map from a weakly pseudoconvex K\"ahler manifold $X$ to an analytic space $Y$.
    Let $E$ be a vector bundle on $X$. If $E$ admits a metric whose curvature is larger than or equal to the pull-back of
    a K\"ahler form on $Y$ in the sense of Nakano, then we have
    $$H^p(Y,R^q\pi_*(K_X\otimes E))=0 \quad \text{for any } p\geq 1 \text{ and } q\geq 0.$$
\end{theorem}

In \cite{QZ26}, the second author and the third author unified the results of Koll\'ar, Ohsawa, Fujino, and Matsumura. We have the following result.

\begin{theorem}[{\cite[Theorem 1.8]{QZ26}}]
    Let $\pi:X\rightarrow Y$ be a proper surjective holomorphic map from a weakly pseudoconvex K\"ahler manifold $X$ to an analytic space $Y$ with a K\"ahler metric $\sigma$.
    Let $L$ be a line bundle on $X$ equipped with a singular Hermitian metric $h$ satisfying $\image\Theta_{L,h}\geq\pi^*\sigma$. Then
    $$H^p(Y,R^q\pi_*(K_X\otimes L\otimes\mathcal{I}(h)))=0 \quad \text{for any } p\geq 1 \text{ and } q\geq0.$$
\end{theorem}
One can see \cite{MZ20,IMW24} for the case $q=0$ and the connection between vanishing theorems and Berndtsson's result on positivity of direct image sheaves.\\

Natural questions are whether the Kawamata-Viehweg type vanishing theorem holds for higher direct images in  Koll\'ar type and whether there is a unified and deeper version of the above important  vanishing theorems on compact K\"ahler manifolds.
Giving a positive answer to both of the questions, we develop some new ideas and methods, and establish a general Kawamata-Viehweg-Koll\'ar-Nadel type vanishing theorem for higher direct images on compact K\"ahler manifolds as follows.

\begin{theorem}[=Theorem \ref{MainThm1'}] \label{MainThm1}
    Let $\pi:X\rightarrow Y$ be a holomorphic surjection from a compact K\"ahler manifold $X$ of dimension $n$ to a compact K\"ahler manifold $Y$ of dimension $m$.
    Let $L$ be a holomorphic Hermitian line bundle over $X$ equipped with a singular Hermitian metric $h$ and $T$ a closed positive $(1,1)$-current on $Y$.
    We assume that $\image\Theta_{L,h}\geq \pi^*T$. Then
    $$H^p(Y,R^q\pi_*(K_X\otimes L\otimes\mathcal{I}(h)))=0 \quad \text{for any } p\geq m-\nd(T)+1 \text{ and } q\geq0. $$

\end{theorem}
We remark that if $T=\image\Theta_{F,h_F}$ for some pseudo-effective line bundle $(F,h_F)$ on a projective manifold,
the definition of numerical dimension $\nd(T)$ of $T$ given by Cao coincides with Tsuji's definition of numerical dimension $\nu_\textup{num}(F,h_F)$. (See \cite[Proposition 4.2]{Cao14}).

The strategy of the proof of Theorem~\ref{MainThm1} is based on recently developed methods and results, by developing the Bochner type technique with special singular hermitian metrics constructed by means of the Calabi-Yau theorem and Demailly's equisingular approximation theorem, and Guan-Zhou's strong openness theorem.
The key points are cohomology class representatives for higher direct images of pseudo-effective line bundles (see Theorem~\ref{thm:construction of representatives})
and to consider the multiplier ideal sheaves of the pull-back of these special metrics by the strong openness property established in \cite{GZ15-a}, see Lemma~\ref{lemma:regularization2}.

As a counterpart to Theorem~\ref{K-V vanishing}, we give the following vanishing theorem over projective manifolds. 
    One can derive Theorem~\ref{K-V vanishing} directly by taking $\pi$ to be the identity in Theorem~\ref{MainThm4}.

\begin{theorem}[=Theorem \ref{MainThm4'}]\label{MainThm4}
    Let $X$ be a compact K\"ahler manifold and $Y$ a projective manifold of dimension $m$. Let $\pi:X\rightarrow Y$ be a holomorphic surjection and $L$ a line bundle over $X$.
    Assume that $\alpha$ is a nef class on $Y$ and $c_1(L)-\pi^*\alpha=\{\theta+\ddbar\psi\}$, where $\theta$ is a smooth real $(1,1)$-form and $\psi$ is a function on $X$ such that $\theta+\ddbar\psi\geq 0$ in the sense of currents.
    Then
    $$H^p(Y,R^q\pi_*(K_X\otimes L\otimes\mathcal{I}(\psi)))=0 \quad \text{for any } p\geq m-\nd(\alpha)+1 \text{ and } q\geq0. $$    
    In particular, if $L$ is a nef line bundle on $Y$, then
    $$H^p(Y,R^q\pi_*K_X\otimes L)=0 \quad \text{for any } p\geq m-\nd(L)+1 \text{ and } q\geq0. $$ 
    
\end{theorem}

When $Y$ is noncompact, we have the following Girbau type theorem, which can be seen as a relative version.

\begin{theorem}[=Theorem \ref{MainThm2'}] \label{MainThm2}
    Let $\pi:X\rightarrow Y$ be a proper holomorphic surjection from a weakly pseudoconvex K\"ahler manifold $X$ of dimension $n$ to a complex manifold $Y$ of dimension $m$.
    Let $L$ be a holomorphic Hermitian line bundle over $X$ equipped with a singular Hermitian metric $h$ and $\sigma$ a continuous real semi-positive $(1,1)$-form on $Y$ satisfying $\sigma^l\neq 0$ everywhere on $Y$.
    We assume that $\image\Theta_{L,h}\geq \pi^*\sigma$. Then
    $$H^p(Y,R^q\pi_*(K_X\otimes L\otimes\mathcal{I}(h)))=0 \quad \text{for any } p\geq m-l+1 \text{ and } q\geq0. $$

\end{theorem}

The first author and the third author removed the K\"ahler condition in Nadel's vanishing theorem on holomorphically convex manifolds in \cite[Theorem 3.8]{MZ19}.
We can also obtain a similar result on higher direct images as a corollary of Theorem~\ref{MainThm1}.

\begin{theorem}[=Theorem \ref{MainThm3'}]\label{MainThm3}
    Let $X$ be a compact manifold in the Fujiki class. Let $\pi:X\rightarrow Y$ be a holomorphic surjection to a compact K\"ahler manifold $Y$ of dimension $m$,
    and let $L$ be a line bundle over $X$ equipped with a singular Hermitian metric $h$ and $T$ a closed positive $(1,1)$-current on $Y$.
    We assume that $\image\Theta_{L,h}\geq \pi^*T$. Then
    $$H^p(Y,R^q\pi_*(K_X\otimes L\otimes\mathcal{I}(h)))=0 \quad \text{for any } p\geq m-\nd(T)+1 \text{ and } q\geq0. $$

\end{theorem}

Our theorems are new even for projective manifolds and unify several previously known vanishing theorems.

For the related topics about higher direct images and vanishing theorems, we refer to \cite{Kaw82,Vie82,Ohs84,Siu84,Kol86-a,Kol86-b,Eno93,Tak95,DP03,Laz04-a,Laz04-b,Cao14,Dem15,Mat16,Fuj17,Fuj18,MZ19,MZ20,FM21,ZZ22,SZ23,IMW24,FW25,QZ26} and references therein.

The remaining parts of this paper are organized as follows. In Section~\ref{Sec:Pre}, we briefly recall some basic preparatory results about the numerical dimension and $L^2$ theory for the $\bar\partial$-operator.
In Section~\ref{Sec:Girbau}, we prove Girbau type vanishing Theorem~\ref{MainThm2}. % first to explain our idea to prove the vanishing Theorem~\ref{MainThm1} in the smooth case.
Then in Section~\ref{Sec:Kawamata-Viehweg}, we prove Theorem~\ref{MainThm1} by the strong openness property and a mass concentration technique for Monge-Amp\`ere
equations. We prove Theorem~\ref{MainThm4} by induction.
In Section \ref{Sec:Application}, we prove Theorem \ref{MainThm3}.

\section{Preliminaries} \label{Sec:Pre}

In this section, we recall and obtain some basic results used in the proofs of the main results in the present paper.
In this paper, we use the terms ``(Cartier) divisors", ``(holomorphic) line bundles", and ``invertible sheaves" interchangeably.

\subsection{$L^2$-spaces and $L^2$-estimates}

Let $(X,\omega)$ be a Hermitian manifold of dimension $n$,
and let $(L,h)$ be a Hermitian holomorphic line bundle over $X$.
Let $D$ be the Chern connection of $(L,h)$.
It can be written in the unique way as $D=D_h'+\bar\partial$.

For $L$-valued $(p,q)$-forms $u$ and $v$, the global inner product
$\llangle u,v\rrangle_{h,\omega}$ is defined as
$$\llangle u,v\rrangle_{h,\omega}:=\int_X\langle u,v\rangle_{h,\omega} dV_\omega,$$
where $\inner{u,v}_{h,\omega}$ denotes the pointwise inner product induced by $h$ and $\omega$, and $dV_\omega$ is the volume form associated with $\omega$.
The $L^2$ space of $L$-valued $(p,q)$-forms with respect to $h$ and $\omega$ is defined by
$$L^{p,q}_{(2)}(X,L)_{h,\omega}:= \{u\ |\ u\ \text{is an $L$-valued $(p,q)$-form with}\  \|u\|_{h,\omega}<+\infty\}.$$

Then $\bar{\partial}$ is a densely defined closed operator on $L^{p,q}_{(2)}(X,L)_{h,\omega}$ with the domain
$$\operatorname{Dom}\bar\partial:=\{u\in L^{p,q}_{(2)}(X,L)_{h,\omega}|\ \bar\partial u\in L^{p,q+1}_{(2)}(X,L)_{h,\omega} \}.$$
Even if $\omega$ is not complete, we always have the orthogonal decomposition (see \cite{DemBig})
$$L^{p,q}_{(2)}(X,L)_{h,\omega}=\overline{\operatorname{Im}\bar\partial}
\oplus\mathscr{H}^{p,q}_{h,\omega}(X,L)
\oplus\overline{\operatorname{Im}\bar\partial^\ast_{h,\omega}},$$
where $\bar\partial^\ast_{h,\omega}$ is the Hilbert adjoint of $\bar\partial$,
and $\mathscr{H}^{p,q}_{h,\omega}(X,L)$ is the set of harmonic forms with respect to $h$ and $\omega$, that is,
$$\mathscr{H}^{p,q}_{h,\omega}(X,L):=\{u\in L^{p,q}_{(2)}(X,L)_{h,\omega}|\ \bar\partial u=0\ and\
\bar\partial^\ast_{h,\omega}u=0\}.$$
In general, the above decomposition also holds when $h$ is a singular metric such that $i\Theta_{L,h}\geq\gamma$ for some smooth $(1,1)$-form $\gamma$ (see \cite{Tak97,Mat19}).

Consider the local $L^2$-space
$$L^{p,q}_{(2,\loc)}(X,L)_{h,\omega}:= \{u\ |\ u \text{ is an $L$-valued $(p,q)$-form with } \|u\|_{K,h,\omega}<+\infty,\  \forall K\Subset X\},$$
where $\|u\|^2_{K,h,\omega}=\int_K |u|^2_{h,\omega} dV_\omega$.
We have the standard isomorphism
$$H^q(X,K_X\otimes L\otimes \mathcal{I}(h))\cong \frac{\operatorname{Ker}\bar\partial}{\operatorname{Im}\bar\partial} \text{ of } L^{n,q}_{(2,\loc)}(X,L)_{h,\omega}.$$

This isomorphism allows us to obtain our vanishing theorems by using the $L^2$ method to solve $\bar\partial$ equations. We need the following useful lemmas.

\begin{lemma}[{\cite[Proposition 3.12]{Dem16}}] \label{lemma:solve equation}
    Let $X$ be a complete K\"ahler manifold equipped with a
    (non necessarily complete) K\"ahler metric $\omega$, and let $(Q,h)$
    be a holomorphic vector bundle over $X$ equipped with a smooth Hermitian metric $h$. Assume that $\tau$ and $A$
    are smooth and bounded positive functions on $ X $ and let
    $$B:=[\tau\sqrt{-1}\Theta_{Q,h}-\sqrt{-1}
    \partial\bar\partial\tau
    -\sqrt{-1}A^{-1}\partial\tau\wedge\bar\partial\tau,\Lambda ],$$
    where $\Lambda$ is the adjoint operator of $L_\omega=\omega\wedge\cdot$.
    Assume that $\delta\geq0$ is a number such that $B+\delta
    \mathrm{I}$ is semi-positive definite everywhere on
    $\Lambda^{n,q}T^*_ X \otimes Q$ for some $q\geq 1$. Then given a form
    $g\in L^{n,q}_{(2)}( X , Q)_{h,\omega}$ such that $\bar\partial
    g=0$ and
    $$\int_ X \langle {(B+\delta
        \mathrm{I})}^{-1}g,g\rangle_{\omega,h} dV_{X,\omega}<+\infty,$$
    there exists an
    approximate solution $u\in L^{n,q-1}_{(2)}( X , Q)_{h,\omega}$
    and a correcting term $v\in L^{n,q}_{(2)}( X , Q)_{h,\omega}$
    such that $\bar\partial  u+\sqrt{\delta}v=g$ and
    $$\int_ X \frac{|u|^2_{\omega,h}}{\tau+A} dV_{X,\omega}+\int_X|v|^2_{\omega,h} dV_{X,\omega}\leq \int_ X
    \langle {(B+\delta \mathrm{I})}^{-1}g,g\rangle_{\omega,h} dV_{X,\omega}.$$
\end{lemma}

In the present paper, we use this lemma in the case that $\tau$ and $A$ are constants. For a more general version of this lemma, we refer to \cite{ZZ18} \cite[Lemma 9.10]{GMY23}.

\begin{lemma}[{\cite[Theorem 1.5]{Dem82}}] \label{lemma:complete metric}
    Let $X$ be a K\"{a}hler
    manifold and $Z$ an analytic subset of $X$. Assume that $\Omega$
    is a relatively compact open subset of $X$ possessing a complete
    K\"{a}hler metric. Then $\Omega\setminus Z$ carries a complete
    K\"{a}hler metric.
\end{lemma}

\begin{lemma}[{\cite[Lemma 6.9]{Dem82}}] \label{lemma:extension}
    Let $\Omega$ be an open subset of
    $\mathbb{C}^n$ and $Z$ be a complex analytic subset of $\Omega$.
    Assume that $u$ is a $(p,q-1)$-form with $L^2_{\loc}$
    coefficients and $g$ is a $(p,q)$-form with $L^1_{\loc}$
    coefficients such that $\bar\partial u=g$ on $\Omega\setminus Z$ (in the sense of currents). Then $\bar\partial u=g$ on
    $\Omega$.
\end{lemma}

By Lemmas~\ref{lemma:solve equation}, Lemma~\ref{lemma:complete metric} and Lemma~\ref{lemma:extension},
it suffices to solve the $\bar\partial$ equation on the complement of an analytic subset with respect to a not necessarily complete K\"ahler metric, when the goal is to solve it on a relatively compact open subset.

\subsection{Regularization and Numerical dimension}

We first recall that for a nef class $\alpha\in H^{1,1}(X,\mathbb{R})$ on a compact K\"ahler manifold $X$, 
the numerical dimension $\nd(\alpha)$ is defined to be $\nd(\alpha)=\max\{k\in\mathbb{N}~|~\alpha^k\neq 0\}$.
For a nef line bundle $L$ on $X$, $\nd(L)$ is defined by $\nd(L)=\nd(c_1(L))$.

One can also define the numerical dimension of some representative of a pseudo-effective class, more precisely, the numerical dimension of closed positive $(1,1)$-currents.

Let us recall the equisingular approximation theorem due to Demailly, Peternell, and Schneider before we give the precise definition of the numerical dimension.

\begin{theorem}[{\cite[Theorem 2.3, Remark 2.5]{DPS01}}] \label{thm:equisingular}
    Let $T=\alpha+\ddbar\varphi$ be a closed $(1,1)$-current on a compact Hermitian manifold $(X,\omega)$, where $\alpha$ is a smooth closed $(1,1)$-form and $\varphi$ is a quasi-plurisubharmonic function.
    Let $\gamma$ be a continuous real $(1,1)$-form such that $T\geq \gamma$. Then one can write $\varphi=\lim_{k\to\infty}\varphi_k$ where

    \begin{itemize}
        \item[(1)] $\varphi_k$ is smooth in the complement $X\setminus Z_k$ of an analytic set $Z_k\subset X$;
        \item[(2)] $\{\varphi_k\}$ is a decreasing sequence, and $Z_k\subset Z_{k+1}$ for all $k$;
        \item[(3)] for every $t>0$, $\int_X e^{-t\varphi}-e^{-t\varphi_k}dV_\omega$ is finite for $k$ large enough and converges to $0$ as $k\to +\infty$;
        \item[(4)] $\mathcal{I}(\varphi_k)=\mathcal{I}(\varphi)$ for all $k$;
        \item[(5)] $T_k=\alpha+\ddbar\varphi_k$ satisfies $T_k\geq \gamma-\varepsilon_k\omega$, where $\lim_{k\to\infty}\varepsilon_k=0$.

    \end{itemize}
\end{theorem}
We remark that Theorem \ref{thm:equisingular} also holds for relatively compact subsets when $X$ is noncompact.
Theorem~\ref{thm:equisingular} was proved by a Bergman kernel method.

Let $T=\alpha+\ddbar\varphi$ be a closed positive current on a compact K\"ahler manifold $(X,\omega)$.
It follows from \cite[Theorem~2.3]{DPS01} that there exists an approximation $\{\varphi_k\}$ of $\varphi$ for $T$ with the following properties:
\begin{itemize}
    \item[(1)] the sequence $\{\varphi_k\}$ converges to $\varphi$ in $L^1$ topology and
    $$\alpha+\ddbar\varphi_k\geq -\tau_k\cdot\omega$$
    for some constants $\tau_k\to0$ as $k\to+\infty$,
    \item[(2)] all the $\varphi_k$ have analytic singularities, and $\varphi_k$ is less singular than $\varphi_{k+1}$, that is, $\varphi_{k+1}\leq\varphi_k+C$ for some constant $C$.
    \item[(3)] for any $\delta>0$ and $m\in\NN$, there exists $k_0(\delta,m)\in\NN$ such that
    $$\mathcal{I}(m(1+\delta)\varphi_k)\subset \mathcal{I}(m\varphi) \quad \text{for every } k\geq k_0(\delta,m).$$
\end{itemize}
This approximation $\{\varphi_k\}$ is called a quasi-equisingular approximation. Moreover, the following property can also be deduced from the proof:
\begin{itemize}
    \item[(4)] for every $t>0$, $\int_X e^{-t\varphi}-e^{-t\max\{\varphi,\varphi_k\}}dV_\omega$ is finite for $k$ large enough and converges to $0$ as $k\to +\infty$.
\end{itemize}

Although such approximations were constructed in \cite[Theorem~2.3]{DPS01} for the local case,
one can easily adapt that construction to a global situation on projective manifolds by using the same techniques (see \cite[Proposition~3.3]{Cao14}, \cite[Proposition~2.1]{GZ15-b}).

We now recall the definition of the numerical dimension.
\begin{definition}[{\cite[Definition 2.5, Definition 3.1]{Cao14}}]
    Let $T_i = \alpha_i+\ddbar\varphi_i$, $i=1,\cdots,k$ be closed positive $(1,1)$-currents on a compact K\"ahler manifold $X$ of dimension $n$.
    Then one can define a cohomological product
    $$\langle T_1,\cdots,T_k \rangle \in H^{k,k}_{\geq 0}(X)$$
    such that for all $u\in H^{n-k,n-k}(X)$,
    $$\langle T_1,\cdots,T_k \rangle \wedge u = \lim_{j\to \infty}\int_X (\alpha_1+\ddbar\varphi_{1,j})_\textup{ac}
    \wedge\cdots\wedge(\alpha_k+\ddbar\varphi_{k,j})_\textup{ac}\wedge u,$$
    where $\{\varphi_{i,j}\}_{j=1}^\infty$ is a quasi-equisingular approximation of $\varphi_i$ and $(\alpha_i+\ddbar\varphi_{i,j})_\textup{ac}$ is the
    absolutely continuous part of the Lebesgue decomposition of the current $\alpha_i+\ddbar\varphi_{i,j}$.
    The above limit exists and does not depend on the choice of the quasi-equisingular approximation.
    The numerical dimension $\nd(T)$ is defined to be the largest $l\in\NN$ such that $\langle T^l\rangle\neq 0$.

\end{definition}

For reader's convenience, we state the strong openness property of multiplier ideal sheaves proved by Guan and Zhou here.

\begin{theorem}[{\cite[Theorem 1.1]{GZ15-a}}] \label{thm:strong openness}
    Let $\varphi$ be a negative plurisubharmonic function on $\Delta^n\subset \mathbb{C}^n$, and let $\varphi_0\not\equiv-\infty$ be a negative plurisubharmonic function on $\Delta^n$.
    Then $\mathcal{I}(\varphi)=\bigcup_{\varepsilon>0}\mathcal{I}(\varphi+\varepsilon\varphi_0)$.
    In particular, $\mathcal{I}(\varphi)=\mathcal{I}_+(\varphi)$.
\end{theorem}

It follows from the strong openness property and \cite[Lemma 5.9, Lemma 5.10]{Cao14} that a closed positive $(1,1)$-current admits a good regularization.

\begin{lemma} \label{lemma:regularization1}
    Let $(Y,\omega_Y)$ be a compact K\"ahler manifold of dimension $m$.
    Let $T=\alpha+\ddbar\varphi$ be a closed positive $(1,1)$-current on $Y$, and let $p\geq m-\nd(T)+1$.
    Then there exists a sequence of quasi-plurisubharmonic functions $\{\varphi_k\}_{k=1}^\infty$ with analytic singularities which satisfy the following properties:
    \begin{itemize}
        \item[(1)] $\mathcal{I}(\varphi_k)=\mathcal{I}(\varphi)$ and $\varphi_k\leq 0$ on $Y$ for all $k\in\mathbb{N}$.

        \item[(2)] Let $\lambda_{1,k}\leq\cdots\leq\lambda_{m,k}$ be the eigenvalues of $\alpha+\ddbar\varphi_k$ with respect to $\omega_Y$.
        Then there exist two sequences $\tau_k\to 0$ and $\varepsilon_k\to 0$ such that
        $$\varepsilon_k\gg \tau_k+\frac{1}{k} \quad \text{and} \quad \lambda_{1,k}(y)\geq -\varepsilon_k-\frac{C}{k}-\tau_k$$
        for all $y\in Y$ and $k$, where $C$ is a constant independent of $k$.

        \item[(3)] We can choose $\beta>0$ and $0<\gamma<1$ independent of $k$ such that for every $k$, there exists
        an open subset $U_k$ of $Y$ satisfying
        $$\Vol(U_k)\leq\varepsilon_k^\beta \quad \text{and} \quad \lambda_{p,k}+2\varepsilon_k\geq\varepsilon_k^\gamma \quad \text{on } Y\setminus U_k.$$

        \item[(4)] There exists a constant $s_1>0$ such that for any local section $f$ of $\mathcal{I}(\varphi)$ over an open subset $V\subset Y$,
        any relatively compact open subset $U\Subset V$ and every $k$, we have
        $$\int_U |f|^2e^{-\varphi_k}dV_{\omega_Y}\leq C_{\|f\|_{L^\infty}}\cdot (\int_U |f|^2e^{-(1+s_1)\varphi}dV_{\omega_Y})^{1/(1+s_1)}<+\infty,$$
        where $C_{\|f\|_{L^\infty}}$ is a constant depending only on $\|f\|_{L^\infty}$.
    \end{itemize}
\end{lemma}

\begin{remark}
    We briefly recall the construction of $\varphi_k$ for reader's convenience.
    As explained after Theorem~\ref{thm:equisingular}, by using standard gluing techniques and the technique of comparing integrals discussed in \cite{DPS01,GZ15-b,GZ15-c},
    one can obtain a quasi-equisingular approximation $\{\widetilde{\varphi}_k\}$ of $\varphi$ such that
    \begin{itemize}
        \item[(i)] $\widetilde{\varphi}_k$ is less singular than $\varphi$;
        \item[(ii)] $\mathcal{I}((1+\frac{2}{k})\widetilde{\varphi}_k)=\mathcal{I}_+(\varphi)=\mathcal{I}(\varphi)$;
        \item[(iii)] for every $t>0$, $\int_Y e^{-t\varphi}-e^{-t\max\{\varphi,\widetilde{\varphi}_k\}}dV_{\omega_Y}$ is finite for $k$ large enough and converges to $0$ as $k\to +\infty$.
    \end{itemize}

    Let $\pi_k:Y_k\rightarrow Y$ be a log resolution of $\varphi_k$.
    Then $\ddbar(\widetilde{\varphi}_k\circ\pi_k)$ is of the form $[E_k]+C^{\infty}$ where $[E_k]$ is a normal crossing $\mathbb{Q}$-divisor.
    Let $h_k$ be a smooth metric on $[E_k]$ such that for any $\delta>0$ small enough,
    $$\pi_k^*\omega_Y+\delta\image\Theta_{-E_k,h_k}$$
    is a K\"ahler form on $Y_k$.
    Then we can solve a Monge-Amp\`ere equation on $Y_k$,
    \begin{align*}
        &\left
        ((\pi_k^*\ddbar(\alpha+\widetilde{\varphi}_k))_\textup{ac}+\varepsilon_k\pi_k^*\omega_Y+ \delta_k\image\Theta_{-E_k,h_k}+\ddbar\psi_{k,\varepsilon_k,\delta_k}
        \right)^m     \\
        &=C(k,\varepsilon_k,\delta_k)\varepsilon_k^{m-\nd(T)}(\pi_k^*\omega_Y+\delta_k\image\Theta_{-E_k,h_k})^m,
    \end{align*}
    with the normalization condition
    $$\sup_{Y_k}(\widetilde{\varphi}_k\circ\pi_k+\psi_{k,\varepsilon_k,\delta_k}+\delta_k\log|E_k|_{h_k})=0.$$
    Set
    $$\varphi_k=(1+\frac{2}{k}-s)\widetilde{\varphi}_k\circ\pi_k+s(\widetilde{\varphi}_k\circ\pi_k+\psi_{k,\varepsilon_k,\delta_k}+\delta_k\log|E_k|_{h_k}),$$
    where $s$ is a constant to be determined later.
    Then by extending it from $Y\setminus\pi_k(E_k)$ to the whole $Y$, $\varphi_k$ induces a quasi-plurisubharmonic function on $Y$.
    Denote it also by $\varphi_k$ for simplicity.
    We have that $\varphi_k$ is more singular than $(1+2/k)\widetilde{\varphi}_k$.

    Note that $\widetilde{\varphi}_k+\psi_{k,\varepsilon_k,\delta_k}+\delta_k\log|E_k|_{h_k}$ induces a quasi-plurisubharmonic function on $Y$.
    By Skoda's uniform integrability theorem, there exists a constant $a>0$ such that
    $$\int_Y e^{-a(\widetilde{\varphi}_k+\psi_{k,\varepsilon_k,\delta_k}+\delta_k\log|E_k|_{h_k})}$$
    is uniformly bounded for all $k$.
    Take $s_1>0$ such that $\mathcal{I}(\varphi)=\mathcal{I}((1+s_1)\varphi)$.
    Take $s>0$ satisfying $s(1+s_1)/s_1\leq a$.

    Then one can check all properties in Lemma~\ref{lemma:regularization1}.
\end{remark}

\begin{remark}
    We remark that the proof of property (4) (\cite[Lemma 5.10]{Cao14}) relies on Skoda's uniform integrability theorem.
    To simplify the presentation and avoid unnecessary technicalities, we do not discuss this theorem in detail here.
    We refer to \cite[Theorem 8.11]{GZ17} and \cite{XZ26} (see also \cite{Sko72}) for the precise statement and more general results for reader's convenience.

\end{remark}

\subsection{Higher direct images of pseudo-effective line bundles}

In this subsection, we recall the construction of cohomology class representatives for higher direct images of pseudo-effective line bundles, as given in \cite{QZ26}.

The following result plays an important role in the proof of the main theorems.

\begin{theorem}[{\cite[Theorem 1.1, Theorem 1.6]{QZ26}}] \label{thm:construction of representatives}
    Let $\pi:X\rightarrow Y$ be a proper surjective holomorphic map from a K\"ahler manifold $X$ of dimension $n$ to an analytic space $Y$ of dimension $m$, and let $(L,h)$ be a pseudo-effective line bundle on $X$.
    \begin{itemize}
        \item[(1)] If $(X,\omega)$ is holomorphically convex, then for any smooth plurisubharmonic exhaustion function $\Phi$, the natural quotient map
        $$\iota: \mathcal{H}^{n,q}(X,L,h,\Phi):=\bigcap_{c\in\mathbb{R}}{\mathscr{H}^{n,q}_{h,\omega}(X_c,L)} \rightarrow H^q(X, K_X\otimes L\otimes \mathcal{I}(h))$$
        is an isomorphism, where $X_c=\{\Phi<c\}$ is the sublevel set given by $\Phi$.
        For any two smooth plurisubharmonic exhaustion functions $\Phi$ and $\Phi'$, we have
        $$\mathcal{H}^{n,q}(X,L,h,\Phi)=\mathcal{H}^{n,q}(X,L,h,\Phi').$$
        For any $u\in \mathcal{H}^{n,q}(X,L,h,\Phi)$, we have
        $$*u\in H^0(X,\Omega_X^{n-q}\otimes L\otimes \mathcal{I}(h)),$$
        where $*$ is the Hodge star operator related to the K\"ahler metric $\omega$.

        \item[(2)] For integers $p,q\geq 0,$ there exists a natural injective homomorphism
        $$\varphi_{p,q}:H^p(Y,R^q\pi_*(K_X\otimes L\otimes\mathcal{I}(h)))\to H^{p+q}(X,K_X\otimes L\otimes\mathcal{I}(h))$$
        defined as follows.

        Denote by $\mathcal{H}$ the inverse map of the natural quotient map $\iota$. Fix a Stein open cover $\mathcal{W}=\{W_i\}_{i\in I}$ of $Y$.
        Denote the open set $W_{i_0}\cap\cdots\cap W_{i_p}$ by $W_{i_0\ldots i_p}$, and the open set $\pi^{-1}(W_{i_0\ldots i_p})$ by $V_{i_0\ldots i_p}$ for simplicity.
        Fix a partition of unity $\{\rho_i\}_{i\in I}$ subordinate to $\mathcal{W}$.

        For a given $p$-cochain $\{\alpha_{i_0\ldots i_p}\}_{i_0\ldots i_p}$, where each $\alpha_{i_0\ldots i_p}$ is a cohomology class in
        $H^q(V_{i_0\ldots i_p},K_X\otimes L\otimes\mathcal{I}(h))$,
        satisfying the cocycle condition:
        $$\delta(\{\alpha_{i_0\ldots i_p}\}_{i_0\ldots i_p}):=
        \{\sum_{k=0}^{p+1}(-1)^k\alpha_{i_0\ldots \hat{i_k}\ldots i_{p+1}}|_{V_{i_0\ldots i_{p+1}}} \}_{i_0\ldots i_{p+1}}=0,$$
        we can define $L$-valued forms on $X$ by
        \begin{equation*}
            \left\{
            \begin{aligned}
                b_{i_0\ldots i_{p-1}}&=\sum_{j\in I}\pi^*\rho_j\mathcal{H}(\alpha_{ji_0\ldots i_{p-1}}), \\
                b_{i_0\ldots i_{p-k}}&=\sum_{j\in I}\pi^*\rho_j\bar\partial b_{ji_0\ldots i_{p-k}}, \\
                b:=\bar\partial b_{i_0}&=\sum_{j\in I}\bar\partial \pi^*\rho_j\wedge\bar\partial b_{ji_0}.
            \end{aligned}\right.
        \end{equation*}
        Then $\varphi_{p,q}$ is given by $\varphi_{p,q}(\{\alpha_{i_0\ldots i_p}\}_{i_0\ldots i_p}):=[b]\in H^{p+q}(X,K_X\otimes L\otimes\mathcal{I}(h)).$
        The map $\varphi_{p,q}$ is well-defined and independent of the choice of Stein open cover.
    \end{itemize}
\end{theorem}
\begin{remark}\label{Rm:representative and torsion free}
    A direct calculation shows that
    \begin{equation} \label{formula:representative}
        b=\sum_{j\in I}\bar\partial \pi^*\rho_j\wedge\bar\partial b_{ji_0}=
        \sum_{i_0,\cdots,i_p}\bar\partial \pi^*\rho_{i_0}\wedge\cdots\wedge\bar\partial \pi^*\rho_{i_{p-1}}\wedge\mathcal{H}(\alpha_{i_0\ldots i_p}).
    \end{equation}

    As an application of (1), we obtained Koll\'ar package in a very general setting. See \cite[Corollary 3.8]{QZ26} for the precise statement.
    In particular, $R^q\pi_*(K_X\otimes L\otimes\mathcal{I}(h))$ is torsion-free for $q\geq 0$ and vanishes if $q>n-m$.
    See also \cite{Fuj13,FM21,LLYZ25,Mat22,MZ19,SZ23,Tak95,ZZ22} for related topics.
\end{remark}

\section{Girbau type theorem} \label{Sec:Girbau}

Before proving our Girbau-Nadel type vanishing theorem on weakly pseudoconvex K\"ahler manifolds, we first establish the following Lemma~\ref{lemma:computation} via direct computations.
This lemma will be used frequently in the paper.

\begin{lemma}\label{lemma:computation}
    Let $\pi:(X,\omega_X)\rightarrow (Y,\omega_Y)$ be a proper holomorphic map between complex manifolds, where $\dim X = n$ and $\dim Y = m$.
    Let $(L,h)$ be a Hermitian holomorphic line bundle on $X$.
    Let $\sigma$ be a continuous real semi-positive $(1,1)$-form on $Y$ and
    $0\leq\lambda_1\leq\cdots\leq\lambda_m$ eigenvalues of $\sigma$ with respect to $\omega_Y$ satisfying $\lambda_p>0$.
    Let $\varphi$ be an $L$-valued $(n,q)$-form on $X$ and $\rho_1,\ldots,\rho_p$ smooth functions on $Y$.

    Consider the operator $B=[\pi^*\sigma,\Lambda_{\omega_X}]$.
    Then at every point $x\in X$,
    there exists a constant $C>0$ such that, in a neighborhood of $x$, we have
    \begin{align*}
        \langle B^{-1}&(\bar\partial \pi^*\rho_1\wedge\cdots\wedge\bar\partial \pi^*\rho_p\wedge  \varphi), (\bar\partial \pi^*\rho_1\wedge\cdots\wedge\bar\partial \pi^*\rho_p\wedge \varphi) \rangle_{h,\omega_X}\\
        &\leq C \frac{1}{\lambda_1+\cdots+\lambda_p}
        |\bar\partial \rho_1\wedge\cdots\wedge\bar\partial\rho_p|^2_{\omega_Y}\cdot |\varphi|^2_{h,\omega}.
    \end{align*}
\end{lemma}
\begin{proof}
    For a given point $x\in X$, we choose a local coordinate $(z_1,\ldots,z_n)$ centered at $x$ such that
    $$\omega_X= \image\sum_{j=1}^n dz_j\wedge d\bar{z}_j \quad \text{and} \quad \pi^*\sigma=\image\sum_{j=1}^n \tau_jdz_j\wedge d\bar{z}_j \text{ at } x.$$
    We choose a local coordinate
    $(t_1,t_2,\ldots,t_m)$ centered at $y=\pi(x)$ such that
    $$\omega_Y= \image\sum_{j=1}^m dt_j\wedge d\bar{t}_j \quad \text{and} \quad \sigma=\image\sum_{j=1}^m \lambda_jdt_j\wedge d\bar{t}_j \text{ at } y.$$
    Write $\pi=(\pi_1,\pi_2,\ldots,\pi_m)$ with respect to these coordinates. We have
    $$\sum_{j=1}^{m}\lambda_j\frac{\partial \pi_j}{\partial z_k} \overline{\frac{\partial \pi_j}{\partial z_l}}
    =\delta_{kl}\tau_k.$$
    In particular,
    $$\tau_k=\sum_{j=1}^m|\frac{\partial \pi_j}{\partial z_k}|^2\lambda_j.$$
    Denote $dz:=dz_1\wedge\cdots\wedge dz_n$, $d\bar{z}_I:=d\bar{z}_{i_1}\wedge\cdots\wedge d\bar{z}_{i_q}$ for an
    ordered multi-index $I=\{i_1<i_2<\cdots<i_q\}$, and $e$ a local frame of $L$ that gives an orthonormal basis at $x$.
    For the local expressions with respect to these coordinates
    $$\varphi=\sum_{|I|=q}\varphi_I e\otimes dz\wedge d\bar{z}_I \quad \text{and} \quad
    \bar\partial \rho_i=\sum_{j=1}^{m}\frac{\partial\rho_i}{\partial\bar{t}_j}d\bar{t}_j, \quad i=1,2,\ldots,p, $$
    we can easily see that
    $$|\varphi|^2_{\omega_X}=\sum_{|I|=q}|\varphi_I|^2 \quad \text{and} \quad
    \bar\partial \pi^*\rho_i=\sum_{k=1}^n\sum_{j=1}^{m}\frac{\partial\rho_i}{\partial\bar{t}_j}\overline{\frac{\partial \pi_j}{\partial z_k}}d\bar{z}_k, \quad i=1,2,\ldots,p.$$
    By direct computations, we have
    \begin{align*}
        \bar\partial \rho_1\wedge\cdots\wedge\bar\partial\rho_p
        &=(\sum_{j=1}^{m}\frac{\partial\rho_1}{\partial\bar{t}_j}d\bar{t}_j)\wedge \cdots \wedge (\sum_{j=1}^{m}\frac{\partial\rho_p}{\partial\bar{t}_j}d\bar{t}_j)      \\
        &=\sum_{1\leq j_1,\ldots,j_p\leq m}\frac{\partial\rho_1}{\partial\bar{t}_{j_1}}\cdots\frac{\partial\rho_1}{\partial\bar{t}_{j_p}}
        d\bar{t}_{j_1}\wedge\cdots\wedge d\bar{t}_{j_p}  \\
        &=\sum_{J=\{1\leq j_1<j_2<\cdots<j_p\leq m\}} \det (\frac{\partial\rho_\alpha}{\partial\bar{t}_{j_\beta}})_{\alpha,\beta=1,\ldots,p} d\bar{t}_J.
    \end{align*}
    We denote $\det (\rho,t_J):=\det (\frac{\partial\rho_\alpha}{\partial\bar{t}_{j_\beta}})_{\alpha,\beta=1,\ldots,p}$ for $J=\{j_1<j_2<\cdots<j_p\}$.
    Then one can see that
    $$\bar\partial \rho_1\wedge\cdots\wedge\bar\partial\rho_p
    =\sum_{|J|=p}\det (\rho,t_J)d\bar{t}_J \quad \text{and}
    \quad |\bar\partial \rho_1\wedge\cdots\wedge\bar\partial\rho_p|^2_{\omega_Y}
    =\sum_{|J|=p}|\det (\rho,t_J)|^2.$$
    Similarly, we denote $\det (\pi_J,z_K):=\det (\frac{\partial \pi_{j_\alpha}}{\partial\bar{z}_{k_\beta}})$
    for $J=\{j_1<j_2<\cdots<j_p\}$ and $K=\{k_1<k_2<\cdots<k_p\}$. Then we have
    \begin{align*}
        \bar\partial \pi^*\rho_1\wedge&\cdots\wedge\bar\partial \pi^*\rho_p
        =(\sum_{k=1}^n\sum_{j=1}^{m}\frac{\partial\rho_1}{\partial\bar{t}_j}\overline{\frac{\partial \pi_j}{\partial z_k}}d\bar{z}_k)
        \wedge \cdots \wedge (\sum_{k=1}^n\sum_{j=1}^{m}\frac{\partial\rho_p}{\partial\bar{t}_j}\overline{\frac{\partial \pi_j}{\partial z_k}}d\bar{z}_k)      \\
        &=\sum_{1\leq j_1,\ldots,j_p\leq m}\frac{\partial\rho_1}{\partial\bar{t}_{j_1}}\cdots\frac{\partial\rho_p}{\partial\bar{t}_{j_p}}
        \cdot\sum_{1\leq k_1,\ldots,k_p\leq n}
        \overline{\frac{\partial \pi_{j_1}}{\partial z_{k_1}}}\cdots\overline{\frac{\partial \pi_{j_p}}{\partial z_{k_p}}}
        d\bar{z}_{k_1}\wedge\cdots\wedge d\bar{z}_{k_p}  \\
        &=\sum_{|J|=|K|=p} \det(\rho,t_J)\cdot\det(\pi_J,z_K)  d\bar{z}_K.
    \end{align*}
    Then, by putting
    $$g_K:=\sum_{|J|=p}\det(\rho,t_J)\cdot\det(\pi_J,z_K),$$
    we have
    $$|g_K|^2\leq|\bar\partial \rho_1\wedge\cdots\wedge\bar\partial\rho_p|^2_{\omega_Y}\cdot (\sum_{|J|=p}|\det(\pi_J,z_K)|^2)$$
    by the Cauchy–Schwarz inequality. Then, by straightforward computations, we obtain
    \begin{align*}
        \langle B^{-1}&(\bar\partial \pi^*\rho_1\wedge\cdots\wedge\bar\partial \pi^*\rho_p\wedge  \varphi), (\bar\partial \pi^*\rho_1\wedge\cdots\wedge\bar\partial \pi^*\rho_p\wedge \varphi) \rangle_{h,\omega_X} \\
        &=\sum_{|L|=p+q}(\sum_{l\in L}\tau_l)^{-1} (\sum_{I\cup K=L}\sgn^{L}_{IK}\varphi_Ig_K) (\sum_{I'\cup K'=L}\sgn^{L}_{I'K'}\varphi_{I'}g_{K'}) \\
        &\leq C_1\sum_{|L|=p+q}(\sum_{l\in L}\tau_l)^{-1}
        \sum_{I\cup K=L}|\varphi_I|^2|g_K|^2 \\
        &\leq C_2\sum_{|L|=p+q}(\sum_{l\in L}\tau_l)^{-1}
        |\varphi|^2_{\omega_X} |\bar\partial \rho_1\wedge\cdots\wedge\bar\partial\rho_p|^2_{\omega_Y}\cdot (\sum_{\substack{K\subset L \\ |J|=p}}|\det(\pi_J,z_K)|^2).
    \end{align*}
    for some constants $C_1,C_2>0$ independent of the point $x$ and the local coordinates.
    The first inequality follows from the fundamental inequality $(\sum_{i=1}^{N}|a_i|)^2\leq 2^{N-1}\sum_{i=1}^{N}|a_i|^2$, and the second inequality follows from the estimate for $g_K$ above.

    Then it is enough to show that for any fixed multi-index $L$, the following inequality holds for some constant $C>0$.

    $$(\lambda_1+\cdots+\lambda_p)(\sum_{\substack{K\subset L \\ |J|=p}}|\det(\pi_J,z_K)|^2)  \leq C\sum_{l\in L}\tau_l.$$

    Note that for any $J=\{1\leq j_1<j_2<\cdots<j_p\leq m\}$, we have
    $$\tau_l=\sum_{j=1}^m|\frac{\partial \pi_j}{\partial z_l}|^2\lambda_j
    \geq\sum_{\alpha=1}^{p}|\frac{\partial \pi_{j_\alpha}}{\partial z_l}|^2\lambda_{j_\alpha}
    \geq\sum_{\alpha=1}^{p}|\frac{\partial \pi_{j_\alpha}}{\partial z_l}|^2\lambda_\alpha.$$
    Thus for any fixed multi-index $K\subset L$,
    $$\sum_{l\in L}\tau_l\geq\sum_{l\in K}\tau_l\geq C\sum_{l\in K}\sum_{J=\{j_1<j_2<\cdots<j_p\}} \sum_{\alpha=1}^{p}
    |\frac{\partial \pi_{j_\alpha}}{\partial z_l}|^2\lambda_\alpha$$
    for some constant $C>0$.
    Therefore, we only need to show that
    $$(\lambda_1+\cdots+\lambda_p)|\det(\pi_J,z_K)|^2\leq C\sum_{l\in K}\sum_{\alpha=1}^{p}
    |\frac{\partial \pi_{j_\alpha}}{\partial z_l}|^2\lambda_\alpha$$
    for some constant $C>0$.
    It is sufficient to show that for any fixed $\alpha$,
    $$|\det(\pi_J,z_K)|^2\leq C\sum_{l\in K}|\frac{\partial \pi_{j_\alpha}}{\partial z_l}|^2.$$
    By Laplace Expansion Theorem, the conclusion holds.
    We remark that the constant $C$ depends on the choice of local coordinates around $x$.

    This completes the proof.
\end{proof}

\begin{theorem}[=Theorem \ref{MainThm2}]\label{MainThm2'}
    Let $\pi:X\rightarrow Y$ be a proper holomorphic surjection from a weakly pseudoconvex K\"ahler manifold $X$ of dimension $n$ to a complex manifold $Y$ of dimension $m$.
    Let $L$ be a holomorphic Hermitian line bundle over $X$ equipped with a singular Hermitian metric $h$ and $\sigma$ a continuous real semi-positive $(1,1)$-form on $Y$ satisfying $\sigma^l\neq 0$ everywhere on $Y$.
    We assume that $\image\Theta_{L,h}\geq \pi^*\sigma$. Then
    $$H^p(Y,R^q\pi_*(K_X\otimes L\otimes\mathcal{I}(h)))=0 \quad \text{for any } p\geq m-l+1 \text{ and } q\geq0. $$
\end{theorem}

\begin{proof}
    We adopt the notation from Theorem \ref{thm:construction of representatives}. It suffices to show that the map $\varphi_{p,q}$ is the zero map.
    To achieve this, we need to find an $L$-valued $(n,p+q-1)$-form $u$ with appropriate $L^2$-estimate such that $\bar\partial u=b$.

    \textbf{Step 1: Basic construction and setting.}

    Let $0\leq\lambda_1\leq\cdots\leq\lambda_m$ be eigenvalues of $\sigma$ with respect to a fixed Hermitian metric $\omega_Y$ on $Y$.
    Since $\sigma^l\neq 0$ everywhere, it follows that $\lambda_{m-l+1}>0$.

    For any sublevel set $X_c=\{\Phi<c\}$ defined by a smooth plurisubharmonic exhaustion function $\Phi$,
    Theorem \ref{thm:equisingular} yields a family of singular Hermitian metrics $\{h_i\}$ on $X_c$ such that
    \begin{itemize}
        \item[$\bullet$] $h_i$ is smooth on $X_c\setminus Z_i$ for some proper subvariety $Z_i$,
        \item[$\bullet$] $\{h_i\}$ is an increasing sequence and $\lim_{i\to\infty}h_i=h,$
        \item[$\bullet$] $\image\Theta_{L,h_i}\geq \pi^\ast\sigma-\delta_i\omega_X$ where $\lim_{i\to\infty}\delta_i=0$.
    \end{itemize}

    Note that $b=\sum_{i_0,\cdots,i_p}\bar\partial \pi^*\rho_{i_0}\wedge\cdots\wedge\bar\partial \pi^*\rho_{i_{p-1}}\wedge\mathcal{H}(\alpha_{i_0\ldots i_p})$ by formula~(\ref{formula:representative}).

    \textbf{Step 2: Solve $\bar\partial$ equations.}

    Consider operators $B_1:=[\pi^*\sigma,\Lambda_{\omega_X}]\geq 0$ and $B_{2,i}:=[\image\Theta_{L,h_i},\Lambda_{\omega_X}]$.
    According to Lemma~\ref{lemma:computation}, near some given point $x\in X$, we have
    \begin{align*}
        \langle B_1^{-1}b,b \rangle_{h_i,\omega_X}dV_{\omega_X} \leq & \sum_{i_0,\cdots,i_p} C
        \frac{1}{\lambda_1+\cdots+\lambda_p}|\bar\partial \rho_{i_0}\wedge\cdots\wedge\bar\partial \rho_{i_{p-1}}|^2_{\omega_Y} \\
        &\cdot |\mathcal{H}(\alpha_{i_0\ldots i_p})|^2_{h,\omega_X} dV_{\omega_X}
    \end{align*}
    for some constant $C>0$.
    Thus for any $K\Subset X_c$, there exists a constant $C(K)>0$ such that
    \begin{align*}
        \langle B_1^{-1}b,b \rangle_{h_i,\omega_X}dV_{\omega_X} \leq & \sum_{i_0,\cdots,i_p}
        C(K) \frac{1}{\lambda_1+\cdots+\lambda_p}|\bar\partial \rho_{i_0}\wedge\cdots\wedge\bar\partial \rho_{i_{p-1}}|^2_{\omega_Y} \\
        &\cdot |\mathcal{H}(\alpha_{i_0\ldots i_p})|^2_{h,\omega_X} dV_{\omega_X}.
    \end{align*}
    Therefore, by suitably choosing a smooth convex increasing function $\chi$ and replacing $h$ and $h_i$ with $\widetilde{h} := he^{-\chi \circ \Phi}$ and $\widetilde{h}_i := h_i e^{-\chi \circ \Phi}$ respectively,
    we may assume that
    $$\int_{X_c}\langle B_1^{-1}b,b \rangle_{\widetilde{h}_i,\omega_X}dV_{\omega_X} \leq C<+\infty,$$
    where $C>0$ is a constant independent of $c$ and $i$.

    Note that $B_{2,i}+\delta_i[\omega_X,\Lambda_{\omega_X}]\geq B_1\geq 0$. We have
    $$\int_{X_c}\langle (B_{2,i}+\delta_i(p+q)\Id_L)^{-1}b,b \rangle_{\widetilde{h}_i,\omega_X}dV_{\omega_X}\leq \int_{X_c}\langle B_1^{-1}b,b \rangle_{\widetilde{h}_i,\omega_X}dV_{\omega_X} \leq C<+\infty.$$
    Then by Lemma \ref{lemma:solve equation}, we can find an approximate solution $u_{c,i}$ and a correcting term $v_{c,i}$ such that
    \begin{align*}
        u_{c,i}\in L^{n,p+q-1}_{(2)}(X_c\setminus Z_i,L)_{\widetilde{h}_i,\omega_X}, \
        v_{c,i}\in L^{n,p+q}_{(2)}(X_c\setminus Z_i,L)_{\widetilde{h}_i,\omega_X},
    \end{align*}
    and satisfy the equality
    $$\bar\partial u_{c,i}+\sqrt{\delta_i(p+q)}v_{c,i}=b \text{ on } X_c\setminus Z_i$$
    and the $L^2$ estimate
    $$\|u_{c,i}\|^2_{\widetilde{h}_i,\omega_X}+\|v_{c,i}\|^2_{\widetilde{h}_i,\omega_X}\leq C.$$
    Note that Lemma \ref{lemma:extension} shows that the equation can be extended to $X_c$.
    Taking a subsequence if necessary, we obtain weak limits $u_c=\lim_{i\to\infty}u_{c,i}$ and $v_c=\lim_{i\to\infty}v_{c,i}$
    such that
    $$\bar\partial u_c=b \text{ on } X_c$$
    and
    $$\|u_c\|^2_{\widetilde{h},\omega_X}+\|v_c\|^2_{\widetilde{h},\omega_X}\leq C.$$
    We thus obtain solutions $u_c$ on $X_c$ with uniform $L^2$ bound.
    Finally, by extracting a weakly convergent subsequence, we obtain $u$ satisfying $$\bar\partial u=b \text{ on } X$$
    and $$\|u\|^2_{\tilde{h},\omega_X}\leq C.$$
    This completes the proof.
\end{proof}

\section{Kawamata-Viehweg type Theorem} \label{Sec:Kawamata-Viehweg}

In this section, we prove our Kawamata-Viehweg-Koll\'ar-Nadel type vanishing theorem on compact K\"ahler manifolds.

We first give an approximation lemma with strong openness applicable to our situation by Lemma~\ref{lemma:regularization1}.

\begin{lemma} \label{lemma:regularization2}
    Let $\pi:(X,\omega_X)\rightarrow (Y,\omega_Y)$ be a holomorphic surjective map from a compact K\"ahler manifold $X$ to a compact K\"ahler manifold $Y$ of dimension $m$.
    Let $T=\alpha+\ddbar\varphi$ be a closed positive $(1,1)$-current on $Y$, and let $p\geq m-\nd(T)+1$.
    Let $\{\varphi_k\}_{k=1}^\infty$ be the sequence of quasi-plurisubharmonic functions constructed in Lemma~\ref{lemma:regularization1}.
    Then besides properties (1)--(4) in Lemma~\ref{lemma:regularization1}, we also have that
    for any quasi-plurisubharmonic function $\psi'$ on $X$,
    $$\mathcal{I}(\psi'+\pi^*\varphi_k)=\mathcal{I}(\psi'+\pi^*\varphi)$$
    for $k$ large enough.

\end{lemma}

\begin{proof}
    From the proof of Lemma~\ref{lemma:regularization1} (see \cite[Lemma 5.9, formula (5.27)]{Cao14}), one can see that
    $\varphi_k=\widetilde{\varphi}_k+\Phi_k$ where $\widetilde{\varphi}_k$ is the quasi-equisingular approximation of $\varphi$ mentioned after Theorem~\ref{thm:equisingular} and $\Phi_k$ is a function on $Y$.
    Here $\{\widetilde{\varphi}_k\}$ satisfies two properties:
    \begin{itemize}
        \item[(i)] $\varphi_k$ is more singular than $\widetilde{\varphi}_k$ for every $k$,
        \item[(ii)] for every $t>0$, $\int_Y e^{-t\varphi}-e^{-t\max\{\varphi,\widetilde{\varphi}_k\}}dV_{\omega_Y}$ is finite for $k$ large enough and converges to $0$ as $k\to +\infty$.
    \end{itemize}

    Note that we can replace $\varphi_k$ in the proof of Lemma~\ref{lemma:regularization1} (4)
    (that is, $\widehat{\varphi}_k$ in \cite[Lemma 5.10]{Cao14}) by $\psi'+\pi^*\varphi_k$,
    since we can also apply Skoda's uniform integrability theorem.
    Therefore, there exists a constant $s_1>0$ such that for $k$ large enough,
    $$\int_U |f|^2e^{-(\psi'+\pi^*\varphi_k)}dV_{\omega_X}\leq C_{\|f\|_{L^\infty}}\cdot \left(\int_U |f|^2e^{-(1+s_1)(\psi'+\pi^*\varphi)}dV_{\omega_X}\right)^{1/(1+s_1)}<+\infty,$$
    where $f$ is a local section of $\mathcal{I}(\psi'+\pi^*\varphi)$ over an open subset $V\subset X$ and $U\Subset V$ is any relatively compact open subset.
    Here we use the strong openness property (Theorem~\ref{thm:strong openness}).
    In particular, $$\mathcal{I}(\psi'+\pi^*\varphi)\subset\mathcal{I}(\psi'+\pi^*\varphi_k).$$

    Conversely, for a local section $f$ of $\mathcal{I}(\psi'+\pi^*\widetilde{\varphi}_k)$ over an open subset $V\subset X$ and any relatively compact open subset $U\Subset V$,
    we have
    \begin{align*}
        \int_U |f|^2e^{-(\psi'+\pi^*\varphi)}&=\int_U |f|^2e^{-(\psi'+\pi^*\widetilde{\varphi}_k)}
        +\int_U |f|^2e^{-\psi'}(e^{-\pi^*\varphi}-e^{-\pi^*\widetilde{\varphi}_k})  \\
        &\leq \int_U |f|^2e^{-(\psi'+\pi^*\widetilde{\varphi}_k)}
        +\int_U |f|^2e^{-\psi'}(e^{-\pi^*\varphi}-e^{-\pi^*\max\{\varphi,\widetilde{\varphi}_k\}}).
    \end{align*}
    For the second term in the right-hand side, by H\"older's inequality, we have
    \begin{align*}
        &\int_U |f|^2 e^{-\psi'}(e^{-\pi^*\varphi}-e^{-\pi^*\max\{\varphi,\widetilde{\varphi}_k\}})   \\
        &\leq \left(\int_U |f|^{2p} e^{-p\psi'}\right)^{1/p}
        \cdot \left(\int_U  (e^{-\pi^*\varphi}-e^{-\pi^*\max\{\varphi,\widetilde{\varphi}_k\}})^q\right)^{1/q},
    \end{align*}
    where $p$ and $q$ are positive constants satisfying $\mathcal{I}(p\psi')=\mathcal{I}(\psi')$ by Theorem \ref{thm:strong openness} and $1/p+1/q=1$.

    We now show that, for this fixed $q$ (independent of $f$ and $k$),
    $$\int_U  (e^{-\pi^*\varphi}-e^{-\pi^*\max\{\varphi,\widetilde{\varphi}_k\}})^q<+\infty$$
    for $k$ large enough.

    Denote $(e^{-\varphi}-e^{-\max\{\varphi,\widetilde{\varphi}_k\}})^q$ by $g_k$ for simplicity.
    Then for any $t>0$,
    $$g_k^t\leq e^{-qt\varphi}-e^{-qt\max\{\varphi,\widetilde{\varphi}_k\}}$$
    by the basic inequality $(a-b)^q\leq a^q-b^q$ for $a>b>0$.
    By property (ii) mentioned above, we know that
    $$\int_Y g_k^t<+\infty \quad \text{for } k \text{ large enough}.$$

    By the well-known coarea formula (see \cite[Theorem~3.2.12]{Fed69}, \cite[Corollary~2.2]{Nic}),
    we know that for any measurable function $g$ on $X$ satisfying $g\geq 0$, we have the equality
    $$\int_X g dV_{\omega_X}=\int_Y\int_{X_y}\frac{g}{J_\pi} dV_{X_y} dV_{\omega_Y},$$
    where $J_\pi$ is square of the determinant of the Jacobi of $\pi$, $X_y=\pi^{-1}(y)$ and $dV_{X_y}$ is the volume form on $X_y$ induced by $\omega_X$.
    It is enough to consider $y\in Y$ such that $y$ is a regular point of $\pi$.
    Note that $J_\pi=e^{\Psi}$ for some plurisubharmonic function $\Psi$ locally.
    There exists a constant $\varepsilon>0$ such that
    $$\int_Y\int_{X_y}\frac{1}{J_\pi^{1+\varepsilon}} dV_{X_y} dV_{\omega_Y}=\int_X \frac{1}{J_\pi^{\varepsilon}}dV_{\omega_X}<+\infty.$$
    Denote $\int_{X_y}\frac{1}{J_\pi} dV_{X_y}$ by $h(y)$.
    Then by H\"older's inequality, we have
    $$h(y)=\int_{X_y}\frac{1}{J_\pi} dV_{X_y}\leq
    \left(\int_{X_y}\frac{1}{J_\pi^{1+\varepsilon}} dV_{X_y}\right)^{\frac{1}{1+\varepsilon}}\cdot
    \left(\int_{X_y} 1 dV_{X_y}\right)^{\frac{\varepsilon}{1+\varepsilon}}.$$
    Note that $\int_{X_y} 1 dV_{X_y}=\int_X \{[\pi^{-1}(y)]\}\wedge \frac{[\omega_X]^{n-m}}{(n-m)!}$ is a constant independent of $y$,
    since the class of the current $\{[\pi^{-1}(y)]\}\in H^{m,m}(X,\mathbb{Z})$ given by analytic subset $\pi^{-1}(y)$ is independent of $y$.
    Denote this constant by $C$.
    Then
    $$\int_Y h^{1+\varepsilon}dV_{\omega_Y}\leq
    C^\varepsilon\int_Y\int_{X_y}\frac{1}{J_\pi^{1+\varepsilon}} dV_{X_y} dV_{\omega_Y}<+\infty.$$
    By coarea formula and H\"older's inequality, we obtain
    \begin{align*}
        \int_X \pi^*g_k dV_{\omega_X}=\int_Y g_k h dV_{\omega_Y}\leq
        \left(\int_Yg_k^{\frac{1+\varepsilon}{\varepsilon}} dV_{\omega_Y} \right)^{\frac{\varepsilon}{1+\varepsilon}}\cdot
        \left(\int_Yh^{1+\varepsilon} dV_{\omega_Y} \right)^{\frac{1}{1+\varepsilon}}.
    \end{align*}
    It follows from the argument above that $\int_X \pi^*g_k dV_{\omega_X}<+\infty$ for $k$ large enough.

    Therefore, we obtain that $ f\in \mathcal{I}(\psi'+\pi^*\varphi)$ for $k$ large enough.
    Since $\varphi_k$ is more singular than $\widetilde{\varphi}_k$, we have
    $$\mathcal{I}(\psi'+\pi^\ast\varphi_k)\subset\mathcal{I}(\psi'+\pi^\ast\widetilde{\varphi}_k)\subset\mathcal{I}(\psi'+\pi^\ast\varphi).$$

    This completes the proof.
\end{proof}
We can now prove our Theorem~\ref{MainThm1} and Theorem~\ref{MainThm4}.
\begin{theorem}[=Theorem \ref{MainThm1}] \label{MainThm1'}
    Let $\pi:X\rightarrow Y$ be a holomorphic surjection from a compact K\"ahler manifold $X$ of dimension $n$ to a compact K\"ahler manifold $Y$ of dimension $m$.
    Let $L$ be a holomorphic Hermitian line bundle over $X$ equipped with a singular Hermitian metric $h$ and $T$ a closed positive $(1,1)$-current on $Y$.
    We assume that $\image\Theta_{L,h}\geq \pi^*T$. Then
    $$H^p(Y,R^q\pi_*(K_X\otimes L\otimes\mathcal{I}(h)))=0 \quad \text{for any } p\geq m-\nd(T)+1 \text{ and } q\geq0. $$
\end{theorem}

\begin{proof}
    We use the notation in Theorem \ref{thm:construction of representatives}. It is enough to show that the map $\varphi_{p,q}$ is actually the zero map.
    We will prove that $[b]=0\in H^{p+q}(X,K_X\otimes L\otimes \mathcal{I}(h))$ by the Hausdorff property of the cohomology group.
    In general, one cannot directly find a solution $u$ satisfying $\bar\partial u=b$.
    Let $\omega_X$ and $\omega_Y$ be K\"ahler metrics on $X$ and $Y$ respectively, such that $\omega_X\geq \pi^*\omega_Y$.

    \textbf{Step 1: Basic construction and setting.}

    Write $h = h_\infty e^{-\psi}$, where $h_\infty$ is a smooth Hermitian metric on $L$ and $\psi$ is a quasi-plurisubharmonic function.
    Write $T=\alpha+\ddbar\varphi$.
    Then we have
    $$\image\Theta_{L,h}=\image\Theta_{L,h_\infty}+\ddbar\psi\geq \pi^*T=\pi^*(\alpha+\ddbar\varphi).$$
    Hence
    $$\image\Theta_{L,h_\infty}+\ddbar(\psi-\pi^*\varphi)\geq \pi^*\alpha.$$
    This implies that $\psi':=\psi-\pi^*\varphi$ is a quasi-plurisubharmonic function on $X$.
    Let $\{\varphi_k\}^\infty_{k=1}$ be the sequence of quasi-plurisubharmonic function on $Y$ obtained in Lemma \ref{lemma:regularization2}.
    Consequently,
    $$\image\Theta_{L,h_\infty}+\ddbar(\psi-\pi^*\varphi+\pi^*\varphi_k)\geq \pi^*(\alpha+\ddbar\varphi_k).$$
    Define new metrics on $L$ by $h_k:=h_\infty e^{-(\psi'+\pi^*\varphi_k)}$.
    By Lemma \ref{lemma:regularization2}, we have $\mathcal{I}(h_k)=\mathcal{I}(h)$.

    Applying Theorem~\ref{thm:equisingular} to $\psi'$, we obtain approximations $\psi_i'$ such that
    \begin{itemize}
        \item[$\bullet$] $\psi_i'$ is smooth on $X\setminus Z_i$ for some proper subvariety $Z_i$;
        \item[$\bullet$] $\{\psi_i'\}$ is a decreasing sequence and $\lim_{i\to\infty}\psi'_i=\psi$;
        \item[$\bullet$] $\image\Theta_{L,h_\infty}+\ddbar\psi_i'\geq \pi^\ast\alpha-\delta_i\omega_X$ where $\lim_{i\to\infty}\delta_i=0$.
    \end{itemize}
    Therefore we obtain new metrics $h_{k,i}:=h_\infty e^{-(\psi_i'+\pi^*\varphi_k)}$ satisfying $h_{k,i}\leq h_k$ and
    $$\image\Theta_{L,h_{k,i}}=\image\Theta_{L,h_\infty}+\ddbar\psi_i'+\pi^*\ddbar\varphi_k\geq \pi^\ast(\alpha+\ddbar\varphi_k)-\delta_i\omega_X.$$

    We now adapt the construction from the proof of Theorem~\ref{MainThm2'}.

    Following the notation in Lemma \ref{lemma:regularization2}, let  $\lambda_{1,k}\leq\cdots\leq\lambda_{m,k}$ denote the eigenvalues of $\alpha+\ddbar\varphi_k$ with respect to $\omega_Y$.
    Then $\lambda_{1,k}\geq -\varepsilon_k-\frac{C}{k}-\tau_k$ on $Y$.
    Define $$\widetilde{\lambda}_{j,k}=\lambda_{j,k}+4\varepsilon_k>0 \quad \text{for } j=1,\cdots,m.$$

    Note that $b=\sum_{i_0,\cdots,i_p}\bar\partial \pi^*\rho_{i_0}\wedge\cdots\wedge\bar\partial \pi^*\rho_{i_{p-1}}\wedge\mathcal{H}(\alpha_{i_0\ldots i_p})$ by formula (\ref{formula:representative}).

    \textbf{Step 2: Solve $\bar\partial$ equations.}

    Consider operators $B_{1,k}:=[\pi^*(\alpha+\ddbar\varphi_k+4\varepsilon_k\omega_Y),\Lambda_{\omega_X}]\geq 0$ and $B_{2,k,i}:=[\image\Theta_{L,h_{k,i}},\Lambda_{\omega_X}]$.
    By Lemma~\ref{lemma:computation}, near some given point $x\in X$, we have
    \begin{align*}
        \langle B_{1,k}^{-1}b,b \rangle_{h_{k,i},\omega_X}dV_{\omega_X} \leq & \sum_{i_0,\cdots,i_p} C
        \frac{1}{\widetilde{\lambda}_{1,k}+\cdots+\widetilde{\lambda}_{p,k}}|\bar\partial \rho_{i_0}\wedge\cdots\wedge\bar\partial \rho_{i_{p-1}}|^2_{\omega_Y} \\
        &\cdot |\mathcal{H}(\alpha_{i_0\ldots i_p})|^2_{h_k,\omega_X} dV_{\omega_X}.
    \end{align*}
    for some constant $C>0$.
    Therefore, there exists a constant $C_k>0$ such that
    $$\int_X\langle B_{1,k}^{-1}b,b \rangle_{h_{k,i},\omega_X}dV_{\omega_X} \leq C_k<+\infty.$$

    Note that $B_{2,k,i}+(4\varepsilon_k+\delta_i)[\omega_X,\Lambda_{\omega_X}]\geq B_{1,k}+\delta_i[\omega_X,\Lambda_{\omega_X}]>0$.
    We have
    $$\int_X\langle (B_{2,k,i}+(4\varepsilon_k+\delta_i)(p+q)\Id_L)^{-1}b,b \rangle_{h_{k,i},\omega_X}dV_{\omega_X}\leq \int_X\langle B_{1,k}^{-1}b,b \rangle_{h_{k,i},\omega_X}dV_{\omega_X} \leq C_k.$$
    Applying Lemma~\ref{lemma:solve equation}, we obtain an approximate solution $u_{k,i}$ and a correcting term $v_{k,i}$ such that
    \begin{align*}
        u_{k,i}\in L^{n,p+q-1}_{(2)}(X\setminus Z_i,L)_{h_{k,i},\omega_X}, \
        v_{k,i}\in L^{n,p+q}_{(2)}(X\setminus Z_i,L)_{h_{k,i},\omega_X},
    \end{align*}
    and satisfy the equality
    $$\bar\partial u_{k,i}+\sqrt{(4\varepsilon_k+\delta_i)(p+q)}v_{k,i}=b \text{ on } X\setminus Z_i$$
    and the $L^2$ estimate
    $$\|u_{k,i}\|^2_{h_{k,i},\omega_X}+\|v_{k,i}\|^2_{h_{k,i},\omega_X}\leq C_k.$$
    Note that Lemma \ref{lemma:extension} shows that the equation can be extended to $X$.
    Up to a subsequence, we can choose weakly convergent subsequences and obtain weak limits $u_k=\lim_{i\to\infty}u_{k,i}$ and $v_k=\lim_{i\to\infty}v_{k,i}$
    such that
    $$\bar\partial u_k+\sqrt{4\varepsilon_k(p+q)}v_k=b \text{ on } X$$
    and
    $$\|u_k\|^2_{h_k,\omega_X}+\|v_k\|^2_{h_k,\omega_X}\leq C_k.$$

    Thus, $[b]=[\sqrt{4\varepsilon_k(p+q)}v_k]\in H^{p+q}(X,K_X\otimes L\otimes \mathcal{I}(h))$, since $\mathcal{I}(h_k)=\mathcal{I}(h)$.
    Moreover, we have the estimate
    $$\|\sqrt{4\varepsilon_k(p+q)}v_k\|^2_{h_k,\omega_X}\leq 4\varepsilon_k(p+q)C_k.$$
    Therefore, it is sufficient to show that $\lim_{k\to\infty}\varepsilon_k C_k=0$.
    The Hausdorff property of the cohomology group then implies $[b]=0$ (see \cite[Lemma 5.8]{Cao14}).

    \textbf{Step 3: Show that $\lim_{k\to\infty}\varepsilon_kC_k=0$.}

    By Lemma~\ref{lemma:regularization2}, we have
    \begin{align*}
        \varepsilon_k C_k &\leq \varepsilon_k \sum_{i_0,\cdots,i_p} \int C
        \frac{1}{\widetilde{\lambda}_{1,k}+\cdots+\widetilde{\lambda}_{p,k}}|\bar\partial \rho_{i_0}\wedge\cdots\wedge\bar\partial \rho_{i_{p-1}}|^2_{\omega_Y} \cdot |\mathcal{H}(\alpha_{i_0\ldots i_p})|^2_{h_k,\omega_X}  \\
        &\leq C \sum\int_{\pi^{-1}(U_k)}|\mathcal{H}(\alpha_{i_0\ldots i_p})|^2_{h_k,\omega_X} +
        \varepsilon_k^{1-\gamma} C \sum \int_{X\setminus \pi^{-1}(U_k)}|\mathcal{H}(\alpha_{i_0\ldots i_p})|^2_{h_k,\omega_X}.
    \end{align*}
    Again by Lemma \ref{lemma:regularization2}, Skoda's uniform integrability and the strong openness property imply that there exist constants $C>0$ and $s_1>0$, such that on any open subset $U\subset X$,
    $$\int_U |\mathcal{H}(\alpha_{i_0\ldots i_p})|^2_{h_k,\omega_X}\leq C (\int_U |\mathcal{H}(\alpha_{i_0\ldots i_p})|^2_{h_\infty,\omega_X}e^{-(1+s_1)\psi})^{1/(1+s_1)}<+\infty.$$
    Here we use the fact that $*\mathcal{H}(\alpha_{i_0\ldots i_p})$ is holomorphic (see Theorem \ref{thm:construction of representatives}).
    Since $\lim_{k\to\infty}\Vol(U_k)=0$ and $\pi$ is proper, we have $\lim_{k\to\infty}\varepsilon_k C_k=0$.

    This completes the proof.
\end{proof}

\begin{theorem}[=Theorem \ref{MainThm4}]\label{MainThm4'}
    Let $X$ be a compact K\"ahler manifold and $Y$ a projective manifold of dimension $m$. Let $\pi:X\rightarrow Y$ be a holomorphic surjection and $L$ a line bundle over $X$.
    Assume that $\alpha$ is a nef class on $Y$ and $c_1(L)-\pi^*\alpha=\{\theta+\ddbar\psi\}$, where $\theta$ is a smooth real $(1,1)$-form and $\psi$ is a function on $X$ such that $\theta+\ddbar\psi\geq 0$ in the sense of currents.
    Then
    $$H^p(Y,R^q\pi_*(K_X\otimes L\otimes\mathcal{I}(\psi)))=0 \quad \text{for any } p\geq m-\nd(\alpha)+1 \text{ and } q\geq0. $$    
    In particular, if $L$ is a nef line bundle on $Y$, then
    $$H^p(Y,R^q\pi_*K_X\otimes L)=0 \quad \text{for any } p\geq m-\nd(L)+1 \text{ and } q\geq0. $$  
\end{theorem}
\begin{proof}
    If $\nd (\alpha)=m$, $\alpha=\{\eta\}$ is a big class, where $\eta$ is a smooth real $(1,1)$-form. 
    Then for every $\delta>0$, there exists a quasi-psh function $\varphi$ on $Y$ satisfying
    $$\max_{y\in Y}\nu(\varphi,y)<\delta \text{ and }\eta+\ddbar\varphi\geq \varepsilon_0\omega_Y$$
    for some constant $\varepsilon_0>0$, where $\nu(\varphi,y)$ is the Lelong number of $\varphi$ at $y$ and $\omega_Y$ is a K\"ahler metric on $Y$.
    Let $h_\infty$ be a smooth metric on $L$.
    Since $c_1(L)-\pi^*\alpha=\{\theta+\ddbar\psi\}$ and $\theta+\ddbar\psi\geq 0$, there exists a smooth function $\psi'$ on $X$, such that
    $$\image\Theta_{L,h_\infty}-\pi^*\eta+\ddbar\psi'+\ddbar\psi=\theta+\ddbar\psi\geq 0.$$
    
    By taking $\delta$ small enough, we have $\mathcal{I}(\psi+\pi^*\varphi)=\mathcal{I}(\psi)$ by the strong openness property. 
    Then we obtain a singular metric $h:=h_\infty e^{-\psi'-\psi}e^{-\pi^*\varphi}$ on $L$ such that 
    $$\image\Theta_{L,h}=\image\Theta_{L,h_\infty}+\ddbar\psi'+\ddbar\psi+\ddbar\pi^*\varphi\geq\pi^*(\eta+\ddbar\varphi)\geq\pi^*\varepsilon_0\omega_Y$$
    and $\mathcal{I}(h)=\mathcal{\psi}$.
    Note that $\nd(\alpha)=\nd(\varepsilon_0\omega_Y)=m$. By applying Theorem~\ref{MainThm1'}, we have 
    $$H^p(Y,R^q\pi_*(K_X\otimes L\otimes\mathcal{I}(\psi)))=0 \quad \text{for any } p\geq 1 \text{ and } q\geq0. $$

    Assume that $\nd (\alpha)<m$. We prove the theorem by induction. The idea is contained in \cite[Proposition 1.9 ]{FM21}.
    Let $H$ be an ample line bundle.
    We take a sufficiently large positive integer $m$ and a
    general divisor $B\in|H^{\otimes m}|$ such that $D=f^{-1}(B)$ is smooth, contains no associated prime of $\mathcal{O}_X/\mathcal{I}(\psi)$, and satisfies $\mathcal{I}(\psi|_D)=\mathcal{I}(\psi)|_D$.
    (We refer to \cite[Corollary 3.11]{FM21}, \cite{MZ23-b,Xia22} for the related topic on Bertini theorem and the restriction formula.)
    
    Let $A=H^{\otimes m}$. We have the following
    short exact sequence:
    \begin{align*}
        0 \rightarrow K_X\otimes L\otimes \mathcal{I}(\psi) \rightarrow K_X\otimes L\otimes \mathcal{I}(\psi)\otimes A 
        \rightarrow K_D\otimes L|_D\otimes \mathcal{I}(\psi|_D)\rightarrow 0.
    \end{align*}
    Since $B$ is a general member of $|H^{\otimes m}|$, we may assume that $B$ contains no associated primes of $R^qf_*(K_X\otimes L\otimes \mathcal{I}(\psi))$ for every $q$.
    Hence, we can obtain the following short exact sequence for every $q$ (see \cite{Man82}):
    \begin{align*}
        0 &\rightarrow R^qf_*(K_X\otimes L\otimes \mathcal{I}(\psi)) \rightarrow R^qf_*(K_X\otimes L\otimes \mathcal{I}(\psi))\otimes A \\
        &\rightarrow R^qf_*(K_D\otimes L|_D\otimes \mathcal{I}(\psi|_D))\rightarrow 0.
    \end{align*}
    Therefore we get an exact sequence 
    \begin{align*}
        H^p(A,R^qf_*(K_D\otimes L|_D\otimes \mathcal{I}(\psi|_D))) &\rightarrow H^{p+1}(Y,R^qf_*(K_X\otimes L\otimes \mathcal{I}(\psi))) \\
        &\rightarrow H^{p+1}(Y,R^qf_*(K_X\otimes L\otimes \mathcal{I}(\psi))\otimes A)
    \end{align*}
    for every $p\geq 0$.
    Since $A$ is ample enough, we have 
    $$H^{p+1}(Y,R^qf_*(K_X\otimes L\otimes \mathcal{I}(\psi))\otimes A)=0$$
    by Theorem~\ref{MainThm1'}.
    Thus the above exact sequence implies that
    $$H^p(A,R^qf_*(K_D\otimes L|_D\otimes \mathcal{I}(\psi|_D))) \rightarrow H^{p+1}(Y,R^qf_*(K_X\otimes L\otimes \mathcal{I}(\psi)))$$
    is surjetive for every $p\geq 0$. Since $\nd(\alpha|_A)\geq \nd(\alpha)$, 
    we obtain the desired conclusion by induction on the dimension of $Y$.
    
    This completes the proof of Theorem \ref{MainThm4'}.
\end{proof}

\section{Non K\"ahler case} \label{Sec:Application}

In this section, we prove Theorem~\ref{MainThm3}. 

\begin{theorem}[=Theorem \ref{MainThm3}]\label{MainThm3'}
    Let $X$ be a compact manifold in the Fujiki class. Let $\pi:X\rightarrow Y$ be a holomorphic surjection to a compact K\"ahler manifold $Y$ of dimension $m$,
    and let $L$ be a line bundle over $X$ equipped with a singular Hermitian metric $h$ and $T$ a closed positive $(1,1)$-current on $Y$.
    We assume that $\image\Theta_{L,h}\geq \pi^*T$. Then
    $$H^p(Y,R^q\pi_*(K_X\otimes L\otimes\mathcal{I}(h)))=0 \quad \text{for any } p\geq m-\nd(T)+1 \text{ and } q\geq0. $$

\end{theorem}

\begin{proof}
    By the assumption, there exist a K\"ahler manifold $\widetilde{X}$ and a proper modification $\widetilde{\pi}:\widetilde{X}\rightarrow X$.
    Then by Theorem \ref{MainThm1}, we have
    $$H^p(Y,R^q(\pi\circ \widetilde{\pi})_*(K_{\widetilde{X}}\otimes \widetilde{\pi}^*L\otimes\mathcal{I}(\widetilde{\pi}^*h)))=0
    \quad \text{for any } p\geq m-\nd(T)+1 \text{ and } q\geq0.$$
    Since multiplier ideal sheaves satisfy the functorial property under modifications (\cite[(5.8) Proposition]{DemSmall}), we have
    $$\widetilde{\pi}_*(K_{\widetilde{X}}\otimes \widetilde{\pi}^*L\otimes\mathcal{I}(\widetilde{\pi}^*h))=K_X\otimes L\otimes\mathcal{I}(h).$$
    Note that $R^q\widetilde{\pi}_*(K_{\widetilde{X}}\otimes \widetilde{\pi}^*L\otimes\mathcal{I}(\widetilde{\pi}^*h))=0$ for $q>0$
    (see Remark~\ref{Rm:representative and torsion free}).
    We have the following isomorphisms:
    \begin{align*}
        R^q(\pi\circ \widetilde{\pi})_*(K_{\widetilde{X}}\otimes \widetilde{\pi}^*L\otimes\mathcal{I}(\widetilde{\pi}^*h))
        &\cong R^q\pi_*R^0\widetilde{\pi}_*(K_{\widetilde{X}}\otimes \widetilde{\pi}^*L\otimes\mathcal{I}(\widetilde{\pi}^*h))  \\
        &\cong R^q\pi_*(K_X\otimes L\otimes\mathcal{I}(h)).
    \end{align*}
    Therefore,
    $$H^p(Y,R^q\pi_*(K_X\otimes L\otimes\mathcal{I}(h)))=0 \quad \text{for any } p\geq m-\nd(T)+1 \text{ and } q\geq 0.$$
\end{proof}

\end{document}